\documentclass[preprint,1pt,times]{elsarticle}

\usepackage{amssymb}
\usepackage{amsthm}
\usepackage{amscd}
\usepackage{amsmath}
\usepackage{amsfonts}
\usepackage{amssymb}
\usepackage{graphicx}
\usepackage{color}
\usepackage{framed}
\usepackage{comment}
\usepackage{chngcntr}

\newtheorem{theorem}{Theorem}

\newtheorem{lemma}[theorem]{Lemma}

\usepackage{mathrsfs}
\usepackage{titletoc}
\biboptions{sort&compress}
\usepackage[breaklinks,colorlinks]{hyperref}
\usepackage[pagewise,mathlines]{lineno}
\usepackage{geometry}
\usepackage{bm}

\usepackage{tikz}
\usepackage{subcaption}
 \usepackage[mathlines]{lineno}
\usepackage{float}

\newcommand{\D}{\Delta}

\newcommand{\p}{\partial}

\renewcommand{\le}{\left}
\newcommand{\ri}{\right}

\makeatletter

\newcommand{\Rmnum}[1]{\expandafter\@slowromancap\romannumeral #1@}
\makeatother

\journal{***}

\begin{document}
	
\begin{frontmatter}		

\title{Topological degree for negative fractional Kazdan--Warner equation \\
on finite graphs}
			
\author[author1]{Yang Liu}\ead{dliuyang@tsinghua.edu.cn}
\author[author2]{Liang Shan}\ead{shanliang@ruc.edu.cn}
\author[author3]{Mengjie Zhang\footnote{*corresponding author}$^{*,}$}\ead{zhangmengjie@sdufe.edu.cn}

	\address[author1]{Yau Mathematical Sciences Center, Tsinghua University, Beijing, 100084, China}	
	\address[author2]{Gaoling School of Artificial Intelligence, Renmin University of China, Beijing, 100872, China}	
	\address[author3]{Shandong Key Laboratory of Blockchain Finance, Shandong University of Finance and Economics, Jinan,  250014, China}

\begin{abstract}

Studies on Kazdan--Warner equations on graphs have grown steadily, yet the fractional case remains insufficiently explored.
Using topological degree theory, this work investigates the fractional Kazdan--Warner equation in the negative case on connected finite graphs, focusing on existence and multiplicity of solutions.
This work not only  extends the earlier result of  S. Liu and Yang (2020) to the fractional setting,
but also provides a concise proof for the work of  Shan and Y. Liu (2025).
\end{abstract}
		
\end{frontmatter}

	\section{Introduction}

Since the foundational work of Kazdan and Warner \cite{K-W-1,K-W-2} in 1974, the problem of prescribing Gaussian curvature on two-dimensional manifolds has been extensively studied.  They established nearly complete solvability results for the corresponding curvature equation on compact and non-compact surfaces.
Of particular interest is the negative case, where the existence of multiple solutions carries geometric and analytic significance.  Let $(\Sigma,g)$ be a closed Riemann surface with negative Euler characteristic, so that  $\int_\Sigma \kappa dv_g<0$, where $\kappa$ denotes the Gaussian curvature of $(\Sigma,g)$. In 1995,  Ding and Liu \cite{D-L} considered the equation
\begin{align}\label{K-W-M}
\Delta_g u = (h+\lambda) e^{2u} - \kappa,
\end{align}
where $\Delta_g$ is the Laplacian on $(\Sigma, g)$,
 $\lambda\in\mathbb{R}$, and $h$ is a nontrivial smooth function satisfying $h\leq 0$. By combining variational methods with the method of upper and lower solutions, they obtained existence results for \eqref{K-W-M}. For further developments and related contributions we refer the reader to \cite{B-S-W,Struwe1,Struwe2,Yang-Zhu}.

The discrete counterpart presents features that differ markedly from the smooth setting, since curvature prescriptions on graphs do not admit a direct geometric interpretation. In the pioneering series of works \cite{A-Y-Y-1,A-Y-Y-2,A-Y-Y-3}, Grigor’yan, Lin and Yang developed a variational framework for analysis on graphs and initiated the study of the Kazdan--Warner equation in the discrete setting,
obtaining results that analogous the classical work of \cite{K-W-1}.
Further progress can be found in  \cite{C-M,G,G-J,Keller-Schwarz,L-S}. In \cite{L-S}, S. Liu and Yang
 studied the multiplicity problem for the negative case of \eqref{K-W-M} on a finite graph  $G=(V,E,\mu,w)$:
 	\begin{align}\label{14}
		-\Delta u=(h+\lambda) e^{2u}-\kappa,
	\end{align}
  where  $\Delta$ denotes the discrete Laplacian,      $\lambda\in\mathbb{R}$,   the function $h :V\rightarrow \mathbb{R} $ satisfies
   $h\not\equiv 0$ and $\max_{V} h=0$,   and      $\kappa:V\rightarrow \mathbb{R} $ satisfies   $\int_V\kappa d\mu<0$.
There has also been substantial progress on nonlinear equations on graphs, including Schr\"{o}dinger-type equations \cite{Z-Z,C-W-Y,Y-Zhao,H-X,Z-L-Y2}, mean field equations \cite{L2,L-Z,Z2}, Chern-Simons-Higgs equations \cite{HS,H-L-Y,H-W-Y,C-H},  and various other related problems \cite{H-L-W,L-W,P-S,W,S-Y-Z-1,G-H-S,H-S}.

Topological degree theory has long served as a central tool in the study of partial differential equations, both in Euclidean settings and on Riemann surfaces, see for example Li \cite{Y-Li}.
 In recent years, it has also proved useful for analyzing analogous equations on finite graphs. The work of Sun and Wang \cite{S-w} marked the first application of degree theory to the Kazdan--Warner equation on a finite graph,
\begin{align*}
-\Delta u = h e^{u} - c,
\end{align*}
where  $h$  is a prescribed function and $c$ is a real number.
Their approach has since inspired further developments, see \cite{L1, L-S-Y, Liu-Yang, Liu-Chen-Wang} and the references therein.

We now turn to a fractional analogue of the Kazdan--Warner equation on graphs.    In \cite{Z-L-Y}, Zhang, Lin and Yang stidied the following equation on a finite graph $G=(V,E,\mu,w)$,  namely
\begin{align}\label{s2:9}
(-\Delta)^{s} u = h e^{u} - c,
\end{align}
where  $(-\Delta)^{s}$ is the discrete fractional Laplacian with $s\in (0,1)$,  $h$  is a function on $V$, and $c\in \mathbb{R}$.  Using variational methods, they established the solvability of \eqref{s2:9}.
Later, Shan and Liu in \cite{S-L1}  extended   the result of \cite{L-S} to the fractional setting,  establishing existence and multiplicity for
\begin{align}\label{15}
(-\Delta)^s u = (h+\lambda)e^{2u} - \kappa,
\end{align}
under the same assumptions on $\lambda$, $h$ and $\kappa$ as in \eqref{14}.\\

In this paper,  we aim to apply the topological degree method to study \eqref{15} when $ \kappa$ is a negative constant.
We establish the existence and multiplicity results in \cite{S-L1} by a different method.  
Our main result is as follows.

	\begin{theorem}\label{s2:t1}
Let $G=(V,E,\mu,w)$ be a connected finite graph with vertex set $V$, edge set $E$, measure $\mu$ and weight $w$.
Let  $s\in (0,1)$, $c,\,\lambda \in \mathbb{R}$,  and set  $h_\lambda=h+\lambda$, where
  $h :V\rightarrow \mathbb{R} $ satisfies $ h\not\equiv 0$ and $\max_{V} h=0$.
Assume $c<0$, then the solvability of  the fractional Kazdan--Warner equation
	\begin{align}\label{1}
		(-\Delta)^s u= h_\lambda e^{2u}-c
	\end{align}
depends on  $\lambda$ as follows:
\begin{enumerate}[(i)]
\item If $\lambda\leq 0$, \eqref{1}  admits a unique solution.
\item  There exists a critical value $\Lambda^*_s\in(0,-\min_{V}h)$ such that
\begin{enumerate}[(a)]
\item if $\lambda\in(0,\Lambda^*_s)$,
   \eqref{1} has at least two distinct solutions;
 \item if  $\lambda=\Lambda^*_s$,      \eqref{1} has   at least one solution;
\item if $\lambda\in(\Lambda^*_s,+\infty)$,  \eqref{1} has  no solution.
\end{enumerate}
\end{enumerate}
	\end{theorem}

The proof of Theorem \ref{s2:t1} is based on the topological degree theory, whose  core idea  is to reduce a complex equation to a simpler one through the homotopy invariance of the topological degree. To explain this strategy, we begin by recalling some basic notations on finite graphs.
Let $C(V)$ be the space of all real-valued functions on  $V$.
For $u\in C(V)$, the integral over $V$ is defined by
\begin{align*}
\int_{V}ud\mu=\sum_{x\in V}u(x)\mu(x) .
\end{align*}
For any $p\in [1,+\infty)$,  the space  ${L^p(V)}$  consists of  functions with finite norm
\begin{align*}
\|u\|_p=\left(\sum_{x\in  V}|u(x)|^p\mu(x)\right)^{ 1/p},
\end{align*}
and  ${L^\infty(V)}$ is equipped with the norm $\|u\|_{\infty}=\sup_{x\in  V}|u(x)|.$
Define
\begin{align*}
B_R=\left\{u \in C(V) :\|u\|_{\infty}<R\right\}
\end{align*}
and  $\p B_R=\{u \in C(V):\|u\|_{\infty}=R\}.$

With these notations, we recall several basic facts about the fractional Laplace operator on  finite graphs, following \cite{L2, Z-L-Y}.
  Let $n$ be the number of vertices in $V$,  and set
$\{ \lambda_i\}_{i=1}^n$ and $\{\phi_i\}_{i=1}^n$ denote the eigenvalues and corresponding orthonormal eigenfunctions of the graph Laplacian $-\Delta$.
For any $s\in (0,1)$  and $u\in  C(V)$, the fractional Laplace operator $(-\Delta)^s$ is written as
\begin{align} \label{e-p}
(-\Delta)^s u(x)  = \frac{1}{\mu(x)}\sum_{y \in  V, \, y \neq  x}W_s(x, y) \left(u(x)-u(y)\right),
\end{align}
where the kernel
\begin{align}\label{Hs}
W_s(x, y) =  -\mu(x)\mu(y)\sum_{i=1}^{n}\lambda_i^s\phi_i(x)\phi_i(y),\quad x \neq y,
\end{align}
is positive and symmetric. In particular,   
the operator satisfies the  spectral property
 \begin{align}\label{phi-1}
( -\Delta)^s \phi_i =\lambda_i^s\phi_i , \quad  i= 1, 2,\cdots,n.
\end{align}
An integration by parts formula also holds: for any $u,v \in C(V)$,
  \begin{align}\label{part}
\int_{ V} v(-\Delta)^s u  d\mu=  \int_{ V}  \nabla^s u \nabla^s v d\mu =\int_{ V}u(-\Delta)^{s} v  d\mu,
\end{align}
where   
\begin{align}\label{N}
\nabla^su \nabla^sv   (x)
  =   \frac{1}{2\mu(x)}\sum_{y \in  V, \, y \neq x} W_s(x, y) \left(u(x)-u(y)\right) \left(v(x)-v(y)\right).
  \end{align}
For later use, we  define the fractional Kazdan--Warner map $\mathcal{F}: C(V)\rightarrow  C(V)$ by
     \begin{align}\label{3}
     	\mathcal{F}(u)=(-\Delta)^s u-h_\lambda e^{2u}+c.
     \end{align}

With the analytic framework in place, we outline the remainder of the paper.
 In  Section \ref{S2}, we  first give an a priori estimate for the solutions of  \eqref{1},  and then compute the topological degree of $\mathcal{F}$:  there exists $R_0>0$  such that  for all $R>R_0$,
	\begin{align*}
		\deg(\mathcal{F},B_R,0)= \left\{\begin{aligned}
&1, &\lambda\leq0,\\
			&0, &\lambda>0.		
\end{aligned}\right.
	\end{align*}
In Section \ref{S3}, we  establish  Theorem \ref{s2:t1} from the degree computation.
 In Section \ref{S4}, we  present  numerical experiments  illustrating how different parameter choices affect the critical value  $ \Lambda_s^* $.

\section{Topological degree}\label{S2}

Applying topological degree theory relies on two key steps.
The first step is to  derive a priori bound on solutions,  as presented in the following lemma.

	\begin{lemma}\label{t2}
		If $u$ solves the equation   \eqref{1}, then  $u$ is uniformly bounded.
	\end{lemma}

	\begin{proof}
Since the graph $G$ is finite, for any fixed $\lambda\in \mathbb{R}$, there exists a positive number $\Lambda_\lambda$ such that
		\begin{align}\label{7}
			\Lambda_\lambda^{-1}\leq |h_{\lambda}(x)|\leq \Lambda_\lambda,\quad\forall x\in V\setminus\{h_{\lambda}=0\}.
		\end{align}

We first show that $u$ has a uniform lower bound.
Let $x_0\in V$ be a vertex where $u$ attains its minimum.
From  \eqref{1} and \eqref{e-p},  we infer
	\begin{align}\label{9}
	\nonumber	h_\lambda(x_0)e^{2u(x_0)}-c&=(-\D)^su(x_0)\\
\nonumber&=\frac{1}{\mu(x_0)}\sum_{y\in V, y\not=x_0}W_s(x_0,y) \left(u(x_0)-u(y) \right)\\
&\leq0.
	\end{align}
Since $c<0$,  it follows that
	\begin{align*}
h_\lambda(x_0)<0.
\end{align*}
Combining this with \eqref{7} and \eqref{9} yields
	\begin{align*}
		e^{2u(x_0)}\geq\frac{c}{h_\lambda(x_0)}\geq-\frac{c}{\Lambda_\lambda}>0,
	\end{align*}
and then we get the desired result
	\begin{align}\label{10}
		\min_{V}u\geq\frac{1}{2} \log  \left(-  c  \Lambda_\lambda ^{-1}  \right).
	\end{align}

To obtain the upper bound,   let $x_1\in V$ be a  vertex where $u$ attains its maximum.
Then  there is
	\begin{align}\label{11}
		h_\lambda(x_1)e^{2u(x_1)}-c=(-\D)^s u(x_1)\geq 0.
	\end{align}
We distinguish three possible cases according to the sign of $h_\lambda(x_1)$.


\textit{Case 1. $h_\lambda(x_1)>0$.}  It is clear that
	\begin{align} \label{2}
\nonumber			h_\lambda(x_1)e^{2u(x_1)}
			&=c+\frac{1}{\mu(x_1)}\sum_{y\in V, y\not=x_1}W_s(x_1,y)\left(u(x_1)-u(y)\right)\\
			&\leq  C_M \left(e^{u(x_1)}-\min_{V}u\right),
	\end{align}
where the constant
	\begin{align}\label{5}
		C_M=\max_{x\in V}\frac{\sum_{y\in V, y\not=x}W_s(x,y)}{\mu(x)}>0.
	\end{align}
Applying \eqref{7}, \eqref{10}, \eqref{2} and Young's inequality, we obtain for any  $\epsilon>0$
\begin{align*}
\nonumber	\Lambda_\lambda^{-1}e^{2u(x_1)}&\leq h_\lambda(x_1)e^{2u(x_1)}\\
\nonumber&\leq C_Me^{u(x_1)}- C_M\min_{V}u \\
				&\leq  \epsilon e^{2u(x_1)}+\frac{C_M^2}{4\epsilon} -\frac{ C_M }{2} \log \left(-  c  \Lambda_\lambda ^{-1}  \right).
\end{align*}
Selecting $\epsilon=(2\Lambda_\lambda)^{-1}$ leads to a uniform upper bound
	\begin{align*}
 \max_{V}u\leq C,
	\end{align*}
	where  the constant $C=C(G,\Lambda_\lambda,c)$. 

\textit{Case 2. $h_\lambda(x_1)<0$.}  It follows from \eqref{7} and \eqref{11} that
	\begin{align*}
\Lambda_\lambda^{-1}e^{2u(x_1)}+c\leq 
-(-\D)^s u(x_1)\leq0,
	\end{align*}
which immediately yields a uniform upper bound for $u$.

\textit{Case 3. $h_\lambda(x_1)=0$.}
If $u$ is constant, then \eqref{e-p} and \eqref{11} give the contradiction
\begin{align*}
0<-c=(-\D)^su(x_1)=0.
\end{align*}
 Then  $u$  is not constant.
Since $u(x_1)=\max_{V}u$,  we obtain
	\begin{align*}
 	-c 
\nonumber  &=\frac{1}{\mu(x_1)}\sum_{y\in V, y\not=x_1}W_s(x_1,y)\left(u(x_1)-u(y)\right)\\
&\geq C_m   \left(\max_{V}u-\min_{V}u\right)>0 ,
	\end{align*}
	where  the constant
\begin{align*}
	C_m=\min_{x\not=y}\frac{W_s(x,y)}{\mu(x)}>0.
\end{align*}
Therefore, there is
	\begin{align}\label{4}
  \max_{V}u \leq  \min_{V}u -\frac{c}{C_m}.
	\end{align}
It remains to show that $\min_V u$  admits an upper bound.

Let $x_0$ be the point where $u$ attains its minimum.  If  $u(x_0)\leq0$, then the conclusion follows immediately.
If  $u(x_0)>0$,  then it follows from \eqref{1}, \eqref{7}, \eqref{4} and $h_\lambda(x_0)<0$ that
	\begin{align*}
\Lambda_\lambda^{-1}e^{2u(x_0)}&\leq-h_\lambda(x_0)e^{2u(x_0)}\\
&=-c+\frac{1}{\mu(x_0)}\sum_{y\in V, y\not=x_0}W_s(x_0,y)(u(y)-u(x_0))\\
&\leq  -c+ C_M \left(u(x_1)-u(x_0)\right)\\
&\leq  -c  \left(1+ \frac{C_M}{C_m}\right) ,
	\end{align*}
where the constant $	C_M$ is given by \eqref{5}. This again yields a bound
\begin{align*}
 \min_{V}u \leq C=C(G, \Lambda_\lambda,c),
\end{align*}
which together with \eqref{4} completes the proof.
\end{proof}

 The second step is to compute the topological degree of the fractional Kazdan--Warner map $\mathcal{F}$� by  its homotopy invariance.

   \begin{lemma}\label{t3}
There exists a positive number $R_0$ such that for all $R> R_0$,
	\begin{align*}
		 \deg \le( \mathcal{F},B_R,0 \ri)= \left\{\begin{aligned}
&1, &\lambda\leq0,\\
&0, &\lambda>0.
\end{aligned}\right.
	\end{align*}
\end{lemma}
	\begin{proof}
We first construct a homotopy needed to  calculate the topological degree.
Define 
	\begin{align*}
		\mathcal{T}(u,t)=(-\Delta)^s u-h_{\lambda,t}e^{2u}+c,
	\end{align*}
for any $u\in C(V)$ and  $t\in[0,1]$,   where
	\begin{align}\label{43}
h_{\lambda,t}=(1-t)\,h_\lambda+t\,\tilde{h}_\lambda \quad \text{with} \quad
\tilde{h}_\lambda(x)= \left\{\begin{aligned}
			&\frac{2|V|}{\min_{V}\mu}  &\text{ if }\quad  h_\lambda(x)>0,\\
			&-1 &\text{ if }\quad h_\lambda(x)\leq0,\end{aligned}\right.
\end{align}
and $|V|=
\sum_{x\in V} \mu(x) .$
It is immediate that $h_{\lambda,0}=h_\lambda$ and
  	\begin{align*}
\mathcal{T}(u,0)=\mathcal{F}(u)\quad \text{and}	 \quad	\mathcal{T}(u,1)=(-\Delta)^s u- \tilde{h}_\lambda e^{2u}+c.
	\end{align*}
 Following the argument in the proof of Lemma \ref{t2}, we derive a uniform bound for all solutions of $\mathcal{T}(u,t)=0$.
That is, if $u_t$ solves $\mathcal{T}(u,t)=0$ for any fixed $t\in [0,1]$,  then
there exists   $R_0>0$  such that
\begin{align*}
\|u_t\|_\infty\leq R_0.
\end{align*}
  Hence, for all $R>R_0$,
\begin{align*}
\left\{\p B_R\times[0,1]\right\}\cap\mathcal{T}^{-1}(\{0\})=\emptyset .
\end{align*}
By the homotopy invariance of the topological degree,  one has
	\begin{align}\label{20}
	\deg\le(\mathcal{F},B_R,0\ri) =	 \deg \le(\mathcal{T}(\cdotp,0),B_R，0\ri)
=\deg\le(\mathcal{T}(\cdotp,1),B_R,0\ri),\quad\forall  R>R_0.
	\end{align}
We now analyze the degree according to the sign of $\lambda$.

 \textit{Case 1. $\lambda\leq0$.} We first determine the solvability of $\mathcal{T}(u,1)=0$.
 		We claim that the unique solution is the constant function
\begin{align*}
u\equiv\log\sqrt{-c}.
\end{align*}
 Otherwise, there exist  two  vertices $x_0\not =x_1$ such that
\begin{align*}
u(x_0)=\min_Vu<u(x_1)=\max_Vu.
\end{align*}
   Since $\max_{x\in V}h(x)=0$ and $\lambda\leq0$,  the definition of $\tilde h_\lambda$ implies $\tilde h_\lambda\equiv -1$,  and then
	\begin{align*}
 -e^{2u(x_0)}\leq c \leq   -e^{2u(x_1)},
 \end{align*}  	which is impossible. Hence the claim holds.

Rewrite $V=\{x_1, x_2,  \cdots,x_n\}$.
 Each function $u\in C(V)$ can be viewed as a column vector
   \begin{align*}
   	\bm{u}=\le(u(x_1),u(x_2),\cdots,u(x_n)\ri)^T.
   \end{align*}
  Let     $\bm{L_s}$  be the graph fractional Laplacian  matrix
satisfying
    \begin{align*}
    \bm{{L_s}} \bm{u}=
    \left(
  (-\Delta)^s u(x_1),
  (-\Delta)^s  u(x_2), \cdots,
  (-\Delta)^s u(x_n) \right)^T,
   \end{align*}
and set  $\bm{I}$ be the $n\times n$ identity matrix.
A direct computation implies that
		\begin{align*}
			D\mathcal{T}(\log\sqrt{-c},1)=\bm{L_s} -2c\bm{I}.
			\end{align*}
From \eqref{phi-1}, the eigenvalues of $(-\D)^s$ satisfy
  \begin{align*}
       	0=\lambda_1^s<\lambda_2^s\leq \cdots\leq\lambda^s_n.
       \end{align*}
 Then the matrix $\bm{L_s}$ is  positive semidefinite. This together with $c<0$ leads to
				\begin{align*}
			\det	D\mathcal{T}(\log\sqrt{-c},1)&=	\det \le(\bm{L_s} -2c\bm{I}\ri)\\
&\geq\det \bm{L_s} +(-2c)^n\det \bm{I}\\
& >0.
				\end{align*}
Therefore, for all $R>R_0$,
		\begin{align*}
		 \deg \le( \mathcal{F},B_R,0 \ri)& =	 \deg \le(\mathcal{T}(\cdotp,1),B_R\ri)\\
&=\textrm{sgn}\det{D\mathcal{T}(\log\sqrt{-c},1)}\\
& =1.
		\end{align*}

\textit{Case 2. $\lambda>0$.}
By definition of $\tilde h_\lambda$ in \eqref{43},  we get
		\begin{align}\label{6}
			\int_{V}\tilde{h}_\lambda d\mu=2|V|\sum_{\{x\in V:h_\lambda(x)>0\}}\frac{\mu(x)}{\min_{V}\mu} -|V|>0.
		\end{align}
From  \eqref{part}, \eqref{N} and  the inequality
 $\left(a-b\right)(e^{-2a}-e^{-2b}) \leq0$  for any  $a, b\in \mathbb{R}$,
we obtain
\begin{align}\label{8}
\int_V e^{-2u}(-\Delta)^s u d\mu=\int_{V}\nabla^s e^{-2u}\nabla^s ud\mu \leq0, \quad \forall u\in C(V).
\end{align}
If  $u$ solves  $\mathcal{T}(u,1)=0$,  then    from $c<0$ and \eqref{8},
	\begin{align*}
  	\int_V\tilde{h}_\lambda d\mu  =	\int_V e^{-2u}(-\Delta)^s u d\mu+ c\int_{V} e^{-2u}d\mu < 0,
	\end{align*}
which  contradicts \eqref{6}. Then $\mathcal{T}(u,1)=0$ has no solution.
	  Using  \eqref{20} and 	 the Kronecker  existence theorem  \cite[Corollary 3.2.4]{C2},  we conclude that
\begin{align*}
 \deg \le(\mathcal{F},B_R,0\ri)=0,\quad \forall R>R_0.
 \end{align*}
This completes the proof.
	\end{proof}

	\section{Proof of Theorem \ref{s2:t1}}\label{S3}
	
In this section,  we establish Theorem \ref{s2:t1} in two subsections.
First, we consider the solvability of the fractional Kazdan--Warner equation \eqref{1} for $\lambda\leq 0$.
Next, we study the case $\lambda>0$,  and determine a critical value $\Lambda_s^*\in(0,-\min_V h)$ such that equation \eqref{1} exhibits the solution structure described in part (ii) of Theorem \ref{s2:t1}.

 \subsection{The case $\lambda\leq 0$}

As a direct application of Lemma \ref{t3}, we obtain the existence and uniqueness of the solution to \eqref{1} when $\lambda\leq 0$.

\begin{proof}[\textbf{Proof of $(\mathrm{\mathbf{i}})$ in Theorem \ref{s2:t1}}]
If $\lambda\leq 0$,   then Lemma \ref{t3} ensures  that there exists  $R_0>0$ such that
	\begin{align*}
	\deg\le(\mathcal{F},B_{R_0},0\ri)\not=0.
	\end{align*}
Hence the equation \eqref{1} admits at least one solution by the Kronecker existence result \cite[Corollary 3.2.4]{C2}.

To prove uniqueness, assume that there exist two distinct solutions   $u\not\equiv\varphi$.
Then we obtain
	\begin{align}\label{23}
			(-\D)^s(u-\varphi) =h_\lambda (e^{2u }-e^{2\varphi }).
	\end{align}
Integrating over $V$ and using $h_\lambda\leq 0$ and \eqref{part}, we get
	\begin{align}\label{24}
 \max_{V}(u-\varphi)>0.
	\end{align}
Let $x_1\in V$ be a  vertex where $u-\varphi$ attains its maximum. Then
	\begin{align}\label{25}
 (-\D)^s (u-\varphi)(x_1)>0.
	\end{align}
On the other hand, \eqref{23},  \eqref{24} and $h_\lambda\leq 0$ imply
	\begin{align*}
	(-\D)^s (u-\varphi)(x_1)=h_\lambda(x_1)(e^{2u(x_1)}-e^{2\varphi(x_1)})\leq0,
	\end{align*}
which contradicts \eqref{25}.  Therefore, the solution is unique.
\end{proof}

	\subsection{The case  $\lambda>0$}

We now consider the case $\lambda>0$. Our strategy is to first construct a local minimum solution, and then employ the topological degree argument  to establish the existence of another solution.   Note that solutions of \eqref{1} correspond to critical points of the energy functional $J_\lambda:C(V)\to \mathbb{R}$ defined by
	\begin{align*}
		J_\lambda(u)=\frac{1}{2}\int_{V}|\nabla^s u|^2d\mu-\frac{1}{2}\int_{V}h_\lambda e^{2u}d\mu+c\int_{V}ud\mu.
	\end{align*}
The first step is to show the following lemma.

	\begin{lemma}\label{l7}
		For sufficiently small $\lambda>0$,  
 the equation \eqref{1} admits a local minimum solution.
	\end{lemma}

	\begin{proof}
Define
\begin{align*}
L_\lambda(u)=(-\Delta)^s u - h_\lambda e^{2u} + c,\quad \forall u\in C(V).
\end{align*}
Since $c<0$, we can choose  a sufficiently large constant $B>0$   such that
\begin{align}\label{28}
L_\lambda(-B)=-h_\lambda e^{-2B}+c<0.
\end{align}
By Theorem \ref{s2:t1}, for any fixed $c<1$, there exists a unique function $\psi$ satisfying
\begin{align*}
(-\Delta)^s \psi = h e^{2\psi} - (c-1).
\end{align*}
Then, for any $0<\lambda<e^{-2\psi}$, we  have
		\begin{align}\label{29}
			 L_\lambda(\psi)=
1-\lambda e^{2\psi}>0.
		\end{align}
 Since the functional  $J_\lambda$  is continuously differentiable with respect to $u$, and the set $\{u:-B\leq u\leq\psi\}$  is compact, there exists $u_\lambda\in C(V)$ such that
\begin{align}\label{30}
J_\lambda(u_\lambda)
=\min_{-B\leq u\leq\psi} J_\lambda(u).
\end{align}

We claim that for sufficiently small $\lambda>0$,
		\begin{align}\label{31}
			-B<u_\lambda<\psi.
		\end{align}
Suppose first that $u_\lambda(x_0)=-B$ for some $x_0\in V$.
Choosing $\epsilon>0$  sufficiently small, we have
		\begin{align*}
			-B\leq  u_\lambda+t\delta_{x_0} \leq\psi,\quad\forall   t\in(0,\epsilon),
		\end{align*}
		where $\delta_{x_0}$ denotes the  Dirac function at  $x_0$.
Using \eqref{28} and \eqref{30}, we obtain
		\begin{align*}
				0&\leq\dfrac{d}{dt}\bigg{|}_{t=0}J_\lambda \left(u_\lambda+t\delta_{x_0}\right)\\
				&=(-\D)^s u_\lambda(x_0)+L_\lambda(-B)(x_0)\\
				&<(-\D)^s u_\lambda(x_0).
		\end{align*}
		However, since $x_0$ is the minimum point of $u_\lambda$, we have $(-\D)^s u_\lambda(x_0)\leq0$,
which yields a contradiction.
Next, assume that $u_\lambda(x_1)=\psi(x_1)$ for some $x_1\in V$, then there exists a sufficiently small $\epsilon>0$ such that
		\begin{align*}
			-B\leq  u_\lambda -t\delta_{x_1} \leq \psi ,\quad\forall  t\in(0,\epsilon).
		\end{align*}
Using \eqref{29} and the minimizing property \eqref{30}, we compute
		\begin{align*}
				0&\leq\dfrac{d}{dt}\bigg{|}_{t=0}J_\lambda\left(u_\lambda-t\delta_{x_1}\right)\\
				&=-(-\D)^s (u_\lambda-\psi)(x_1)-L_\lambda(\psi)(x_1)\\
				&<-(-\D)^s (u_\lambda-\psi)(x_1),
		\end{align*}
		which contradicts the facts that $x_1$ is the maximum point of  $u_\lambda-\psi$ and $(-\D)^s (u_\lambda-\psi)(x_1)\geq0$.
  Therefore,    the claim \eqref{31} follows.

 	Finally, \eqref{30} and \eqref{31} imply that  $u_\lambda$ is a local minimum critical point of $J_\lambda$,
and then $u_\lambda$ is a solution of equation  \eqref{1}.
	\end{proof}

Next, we show that the existence of a solution at some parameter propagates to smaller values.

\begin{lemma}\label{l8}
If $L_{\lambda^*}(u)=0$ admits a solution for some $\lambda^*>0$, then for any  $0<\lambda<\lambda^*$,  the equation \eqref{1} has a local minimum solution.
\end{lemma}

	\begin{proof}
Let $u_{\lambda^*}$ be a solution of  $L_{\lambda^*}(u)=0$.
Then   for any  $0<\lambda<\lambda^*$,
\[
L_\lambda(u_{\lambda^*})
>L_{\lambda^*}(u_{\lambda^*})=0.
\]
Combining this with \eqref{28}  and repeating  the same arguments of Lemma \ref{l7}, we obtain  the desired local minimum  $u_\lambda$  satisfying
		\begin{align*}
			J_\lambda(u_\lambda)
=\min_{-B<u<u_{\lambda^*}}J_\lambda(u).
		\end{align*}
This completes the proof.
	\end{proof}

Lemmas \ref{l7} and \ref{l8} allow us to define the critical parameter
	\begin{align}\label{33}
		\Lambda^*_s=\sup\le\{\lambda>0:J_\lambda\ \mathrm{has\ a\ local\ minimum\ critical\ point}\ri\}.
	\end{align}
We now show that $\Lambda_s^*$ has a uniform upper bound depending only on $h$.

	\begin{lemma}\label{l9}
		For any $s\in(0,1)$,  there holds
\begin{align*}
0<\Lambda^*_s<-\min_{V}h.
\end{align*}
	\end{lemma}
	\begin{proof}
Let $u$ be a local minimum solution of \eqref{1}.
		Integrating  over $V$ and using $c<0$, we obtain
			\begin{align*}
				\int_{V}h_\lambda e^{2u}d\mu=c|V|<0,
				\end{align*}
which implies
\begin{align*}
\min_{V}h_\lambda=\min_{V}h+\lambda<0.
\end{align*}
The upper bound then follows immediately from the definition \eqref{33}.
	\end{proof}


At the end of this section,   we complete the proof of Theorem \ref{s2:t1}.

\begin{proof}[\textbf{Proof of $(\mathrm{\mathbf{ii}})$  in Theorem \ref{s2:t1}}]
For $0<\lambda<\Lambda_s^*$,
let $u_\lambda$ denote the unique local minimum critical point of $J_\lambda$, and set $a=J_\lambda(u_\lambda)$.
Let $U$ be a neighborhood of $u_\lambda$,  and
\begin{align*}
J_\lambda^a=\{u\in C(V):J_\lambda(u)\leq a\}.
\end{align*}
From  \cite[Chapter 1]{C},  the $q$-th critical group of $J_\lambda$ at $u_\lambda$ is defined by
 		\begin{align*}
			C_q\left(u_\lambda,J_\lambda\right)=H_q\left(J_\lambda^a\cap U,  (J_\lambda^a\setminus\{u_\lambda\} )\cap U; \mathbb{Z}\right).
		\end{align*}
where  $H_q$ denotes the relatively singular homology  with coefficients  in $\mathbb{Z}$.
By the excision property,  $H_q$ is independent of the choice of  $U$.
A direct computation yields
	\begin{align*}
			C_q(u_\lambda,J_\lambda)
		 =H_q(\{u_\lambda\}, \emptyset; \mathbb{Z}) 	
			 =\delta_{q0}\mathbb{Z}.
	\end{align*}
	Following the proof  of Lemma \ref{t2}, we   deduce that $J_\lambda$ satisfies the Palais-Smale condition.
 Since
		\begin{align*}
			DJ_\lambda(u)=(-\D)^s u-h_\lambda e^{2u}+c=\mathcal{F}(u),
		\end{align*}
	where $\mathcal{F}$ is given in \eqref{3}.
From  \cite[Theorem 3.2]{C}  and 	Lemma \ref{t3},  for any $R> R_0$,
		\begin{align*}
 \deg \le( \mathcal{F},B_R,0 \ri) = \deg \le(DJ_\lambda,B_R,0\ri)
=1,
       \end{align*}
    which contradicts  	\begin{align*} \deg \le( \mathcal{F},B_R,0 \ri)=0 \end{align*}
    derived from Lemma \ref{t3}. Therefore,  $J_\lambda$ has another distinct  critical point, which means that the equation \eqref{1} admits at least two distinct solutions.

Next, for  $\lambda=\Lambda^*_s$,
let  $\{ \lambda_k^*\}$  be an increasing sequence converging to  $\Lambda_s^*$ as $k\rightarrow +\infty$.
For each $k$, Lemma \ref{l8} together with \eqref{33} ensures the existence of a solution $u_{\lambda_k^*}$  to \eqref{1}.
Lemma \ref{t2}  further implies that   the sequence $\{u_{\lambda_k^*}\}$ is uniformly bounded in $C(V)$.
Up to a subsequence,  $u_{\lambda_k^*}$ converges uniformly to a function  $u^*$  as $k\rightarrow +\infty$.
And the limit $u^*$  solves  \eqref{1} for $\lambda=\Lambda_s^*$.
Hence a solution exists at the critical  value.

Finally, for  $\lambda\in(\Lambda^*_s,+\infty)$,   nonexistence follows directly from the definition of  $\Lambda^*_s$.
Otherwise,  \eqref{1} admits a solution for some $\lambda^*>\Lambda_s^*$.
  Then Lemma \ref{l8}  implies the existence of a local minimum solution for any $\lambda\in(\Lambda_s^*,\lambda^*)$, which contradicts  \eqref{33}. This completes the proof.
\end{proof}

\section{Numerical Experiments}\label{S4}

In the previous sections, we focused on establishing the existence of the critical number $ \Lambda_s^* $ as stated in Theorem \ref{s2:t1}. In this section, we present experimental results that investigate the effects of different parameter settings on $ \Lambda_s^* $. We aim to explore the variation of $ \Lambda_s^* $ as the value of $ c $ changes. We use Gurobi to solve this problem and summarize the results. These results provide a clearer and more intuitive understanding of how various configurations influence $ \Lambda_s^* $.
In the following, we first introduce the notations and the settings used in our experiments, and then we present and analyze the experimental results.

For any finite graph $G=(V,E,\mu,w)$,   we  rewrite
$
 V=\{x_1, x_2,  \cdots,x_n\}
$
 with $n=\sharp V$,
define
the  measure matrix  of  $G$ as
          \begin{align*}
 \bm{U}=\text{diag} \left(\mu(x_1), \mu(x_2),\cdots, \mu(x_n) \right) .
   \end{align*}
   and assign the  adjacency matrix of  $G $ as $  \bm{A}=(A_{ij})_{n\times n}  $   with
       \begin{align*}
A_{ij}=   \left\{\begin{aligned}
    &w_{ij},  &\text{ if } x_j\sim x_i,\\
    &0, & \text{else}. \quad \ \ \ \
    \end{aligned}\right.
    \end{align*}
    Each function $u:V\rightarrow\mathbb{R}$ can be viewed as a column vector
       \begin{align*}	\bm{u}=\le(u(x_1),u(x_2),\cdots,u(x_n)\ri)^T.\end{align*}
   Define the  fractional weight matrix  as  $ {\bm{W_s}=(W_{s,ij})_{n\times n}}$ with
    \begin{align*}
W_{s,ij}=   \left\{\begin{aligned}
    &-W_s(x_i, x_j),  &\text{ if }  j\neq  i,\\
    &  \sum_{1\leq j\leq n, j\neq i}W_s(x_i, x_j)  , &\text{ if }  j= i,
    \end{aligned}\right.
    \end{align*}
 where   $s\in (0,1)$ and
 $ W_s(x,y)$ is given as in \eqref{Hs}.  Let
 \begin{align*}
  {\bm{L_s}=\bm{U}^{-1} \bm{W_s}   }
  \end{align*}
   be the graph fractional Laplacian  matrix
satisfying
    \begin{align*}
    \bm{{L_s}} \bm{u}= (-\Delta)^s {\bm{u}} =
    \left(
  (-\Delta)^s u(x_1),
  (-\Delta)^s  u(x_2), \cdots,
  (-\Delta)^s u(x_n) \right)^T.
   \end{align*}
It is clear that
      \begin{align*}
      {\bm{L_s}} =      {\bm{U}^{-1}}\bm{W_s} = \left(\bm{U}^{-\frac{1}{2}} \bm{U}^{\frac{1}{2}}\right)
      {\bm{U}^{-1}}\bm{W_s}
     \left( \bm{U}^{-\frac{1}{2}} \bm{U}^{\frac{1}{2}}\right)
       = \bm{U}^{-\frac{1}{2}}
            \left( { \bm{U}^{-\frac{1}{2}}\bm{W_s}      \bm{U}^{-\frac{1}{2}} }\right)\bm{U}^{\frac{1}{2}}.
       \end{align*}
Since the matrix   $ { \bm{U}^{-\frac{1}{2}}\bm{W_s}      \bm{U}^{-\frac{1}{2}} }$ is a real symmetric matrix,   the graph fractional Laplacian matrix  $   {\bm{{L_s}}} $  can be similarly diagonalized.
From \eqref{phi-1}, we derive
  \begin{align*}
  	\bm{L_s}\bm{\phi}_i=\lambda_i^s \bm{\phi}_i \quad \text{ with } \quad \bm{\phi}_i=({\phi}_i(x_1),{\phi}_i(x_2),\cdots,{\phi}_i(x_n))^T
  	 \end{align*}
  for all $  i  = 1, 2, \cdots , n.$
 Therefore,  we obtain that  $\bm{\phi}_1,\, \bm{\phi}_2,\,\cdots, \,\bm{\phi}_n$ are orthonormal eigenvalues corresponding to the eigenvalues  $\lambda_1^s, \lambda_2^s,   \cdots \lambda_n^s$  of $\bm{L_s}$.
  Here the orthonormality concerns the inner product concerning the measure $\mu$: 
    \begin{align*}
     \langle \bm{\phi}_i, \bm{\phi}_j\rangle= \bm{\phi}_i^T \bm{U} \bm{\phi}_j =
     \sum_{k=1}^{n}\phi_i(x_k)\phi_j(x_k)   \mu(x_k)  =\left\{\begin{aligned}    &1,&{\rm if }\ \ i=j,\\    &0,&{\rm if }\ \ i\not=j.    \end{aligned}\right.
      \end{align*}
If we define
\begin{align*}
\Phi=  \left(\bm{\phi}_1, \bm{\phi}_2,\cdots, \bm{\phi}_n\right) \ \text{  and } \  \Lambda_s  =\text{diag}  ( \lambda^s_1,  \lambda_2^s,\cdots,  \lambda_n^s ) ,
\end{align*}
     then
      \begin{align*}
      \bm{L_s}=\Phi \Lambda_s \Phi^{-1}  \ \text{ and }\   \Phi^{T} \bm{U} \Phi=\bm{I} .
      \end{align*}

Therefore, the problem of solvability of equation  \eqref{1} is transformed into the solvability of the following   problem
\begin{align*}
    {L_s} u_i=  \left ( h_i  + \lambda  \right)e^{2u_i}    -c  ,
\end{align*}
where $u_i$, ${L_s}u_i$ and $h_i$ denote  the $i$-th component of $\bm{u}$, $\bm{{L_s}} \bm{u}$ and $\bm{h}=(h(x_1),h(x_2),\cdots,h(x_n))^T $ respectively.\\

In our experiments, we construct five types of undirected weighted graphs, as illustrated in Figure \ref{fig:1}.
The structure and parameter settings for each type are as follows:
\begin{itemize}
    \item Type 1: This graph consists of  $2$ vertices connected by a single edge. The measure matrix is $ \bm{U_1} = \mathrm{diag}(1, 0.5) $, the adjacency matrix is $ \bm{A_1} = \begin{pmatrix} 0 & 2 \\ 2 & 0 \end{pmatrix} $, and the function is $ \bm{h} = (0, -0.5) $ (see Figure \ref{fig:1.1}).

    \item Type 2: This graph consists of $3$ vertices and $2$ edges. The measure matrix is   $\bm{U_2} = \mathrm{diag}(1, 0.5,$ $0.25)$,  the adjacency matrix is   $ \bm{A_2} = \begin{pmatrix} 0 & 2 & 1 \\ 2 & 0 &0 \\ 1 & 0 & 0 \end{pmatrix} $, and the function is  $ \bm{h} = (0, -0.5, -2) $.  As shown in Figure \ref{fig:1.2}.

    \item Type 3: This is a complete graph with $3$ vertices.  Set the measure matrix $ \bm{U_3} = \mathrm{diag}(1, 0.5,0.25) $, the adjacency matrix $ \bm{A_3} = \begin{pmatrix} 0 & 2 & 1 \\ 2 & 0 &3 \\ 1 & 3 & 0 \end{pmatrix} $, and the  function $ \bm{h} = (0, -0.5, -2) $. As shown in Figure \ref{fig:1.3}.

    \item Type 4: This graph consists of $5$ vertices and $5$ edges. The measure matrix is $ \bm{U_4} = \mathrm{diag}(1, 0.5,$ $0.25,  0.125,0.0625) $, the adjacency matrix is $ \bm{A_4} = \begin{pmatrix} 0 & 0 & 3 & 4 & 0 \\ 0 & 0 & 0 & 2 & 1\\ 3 & 0 & 0 & 0 & 5 \\ 4 & 2 & 0 & 0 & 0\\ 0& 1& 5& 0 & 0\end{pmatrix} $, and  the function is $ \bm{h} = (0, -0.5, -2, -1, -1) $ (see Figure \ref{fig:1.4}).

    \item Type 5:  This is a complete graph with $5$ vertices.  Let $ \bm{U_5} = \mathrm{diag}(1, 0.5,0.25,$ $0.125,0.0625) $, $ \bm{A_5} = \begin{pmatrix} 0 & 1 & 3 & 4 & 1 \\ 1 & 0 & 1 & 2 & 1\\ 3 & 1 & 0 & 1 & 5 \\ 4 & 2 & 1 & 0 & 1\\ 1& 1& 5& 1 & 0\end{pmatrix} $  and $ \bm{h} = (0, -0.5, -2, -1, -1) $ (see Figure \ref{fig:1.5}).
\end{itemize}

\begin{figure}
    \centering
    \begin{subfigure}[b]{0.3\linewidth}
        \centering

        \tikzset{every picture/.style={line width=0.75pt}} 

        \begin{tikzpicture}[x=0.75pt,y=0.75pt,yscale=-1,xscale=1]

        \draw  [fill={rgb, 255:red, 0; green, 0; blue, 0 }  ,fill opacity=1 ] (109.5,160.42) .. controls (109.5,159.22) and (110.47,158.25) .. (111.67,158.25) .. controls (112.86,158.25) and (113.83,159.22) .. (113.83,160.42) .. controls (113.83,161.61) and (112.86,162.58) .. (111.67,162.58) .. controls (110.47,162.58) and (109.5,161.61) .. (109.5,160.42) -- cycle ;
        \draw    (111.67,160.42) -- (216,160.58) ;
        \draw  [fill={rgb, 255:red, 0; green, 0; blue, 0 }  ,fill opacity=1 ] (216,160.58) .. controls (216,159.39) and (216.97,158.42) .. (218.17,158.42) .. controls (219.36,158.42) and (220.33,159.39) .. (220.33,160.58) .. controls (220.33,161.78) and (219.36,162.75) .. (218.17,162.75) .. controls (216.97,162.75) and (216,161.78) .. (216,160.58) -- cycle ;

        \draw (100.33,166) node [anchor=north west][inner sep=0.75pt]   [align=left] {$\displaystyle x_{1}$};
        \draw (212.33,165.5) node [anchor=north west][inner sep=0.75pt]   [align=left] {$\displaystyle x_{2}$};
        \draw (157.33,144) node [anchor=north west][inner sep=0.75pt]   [align=left] {2};

        \end{tikzpicture}

        \caption{Type 1.}
        \label{fig:1.1}
    \end{subfigure}
    \hfill
    \begin{subfigure}[b]{0.3\linewidth}
        \centering

        \tikzset{every picture/.style={line width=0.75pt}} 

        \begin{tikzpicture}[x=0.75pt,y=0.75pt,yscale=-1,xscale=1]

        \draw  [fill={rgb, 255:red, 0; green, 0; blue, 0 }  ,fill opacity=1 ] (109.5,160.42) .. controls (109.5,159.22) and (110.47,158.25) .. (111.67,158.25) .. controls (112.86,158.25) and (113.83,159.22) .. (113.83,160.42) .. controls (113.83,161.61) and (112.86,162.58) .. (111.67,162.58) .. controls (110.47,162.58) and (109.5,161.61) .. (109.5,160.42) -- cycle ;
        \draw    (111.67,160.42) -- (216,160.58) ;
        \draw  [fill={rgb, 255:red, 0; green, 0; blue, 0 }  ,fill opacity=1 ] (216,160.58) .. controls (216,159.39) and (216.97,158.42) .. (218.17,158.42) .. controls (219.36,158.42) and (220.33,159.39) .. (220.33,160.58) .. controls (220.33,161.78) and (219.36,162.75) .. (218.17,162.75) .. controls (216.97,162.75) and (216,161.78) .. (216,160.58) -- cycle ;
        \draw    (111.67,160.42) -- (166.67,85.42) ;
        \draw  [fill={rgb, 255:red, 0; green, 0; blue, 0 }  ,fill opacity=1 ] (165,85.58) .. controls (165,84.39) and (165.97,83.42) .. (167.17,83.42) .. controls (168.36,83.42) and (169.33,84.39) .. (169.33,85.58) .. controls (169.33,86.78) and (168.36,87.75) .. (167.17,87.75) .. controls (165.97,87.75) and (165,86.78) .. (165,85.58) -- cycle ;

        \draw (99.83,166) node [anchor=north west][inner sep=0.75pt]   [align=left] {$\displaystyle x_{1}$};
        \draw (212.33,165.5) node [anchor=north west][inner sep=0.75pt]   [align=left] {$\displaystyle x_{2}$};
        \draw (161.33,144) node [anchor=north west][inner sep=0.75pt]   [align=left] {2};
        \draw (157.33,67.5) node [anchor=north west][inner sep=0.75pt]   [align=left] {$\displaystyle x_{3}$};
        \draw (128.33,108) node [anchor=north west][inner sep=0.75pt]   [align=left] {1};

        \end{tikzpicture}

        \caption{Type 2.}
        \label{fig:1.2}
    \end{subfigure}
    \hfill
    \begin{subfigure}[b]{0.25\linewidth}

        \tikzset{every picture/.style={line width=0.75pt}} 

        \begin{tikzpicture}[x=0.75pt,y=0.75pt,yscale=-1,xscale=1]

        \draw  [fill={rgb, 255:red, 0; green, 0; blue, 0 }  ,fill opacity=1 ] (109.5,160.42) .. controls (109.5,159.22) and (110.47,158.25) .. (111.67,158.25) .. controls (112.86,158.25) and (113.83,159.22) .. (113.83,160.42) .. controls (113.83,161.61) and (112.86,162.58) .. (111.67,162.58) .. controls (110.47,162.58) and (109.5,161.61) .. (109.5,160.42) -- cycle ;
        \draw    (111.67,160.42) -- (216,160.58) ;
        \draw  [fill={rgb, 255:red, 0; green, 0; blue, 0 }  ,fill opacity=1 ] (216,160.58) .. controls (216,159.39) and (216.97,158.42) .. (218.17,158.42) .. controls (219.36,158.42) and (220.33,159.39) .. (220.33,160.58) .. controls (220.33,161.78) and (219.36,162.75) .. (218.17,162.75) .. controls (216.97,162.75) and (216,161.78) .. (216,160.58) -- cycle ;
        \draw    (111.67,160.42) -- (166.67,85.42) ;
        \draw  [fill={rgb, 255:red, 0; green, 0; blue, 0 }  ,fill opacity=1 ] (165,85.58) .. controls (165,84.39) and (165.97,83.42) .. (167.17,83.42) .. controls (168.36,83.42) and (169.33,84.39) .. (169.33,85.58) .. controls (169.33,86.78) and (168.36,87.75) .. (167.17,87.75) .. controls (165.97,87.75) and (165,86.78) .. (165,85.58) -- cycle ;
        \draw    (167.17,85.58) -- (218.17,160.58) ;

        \draw (99.83,166) node [anchor=north west][inner sep=0.75pt]   [align=left] {$\displaystyle x_{1}$};
        \draw (212.33,165.5) node [anchor=north west][inner sep=0.75pt]   [align=left] {$\displaystyle x_{2}$};
        \draw (160.33,145) node [anchor=north west][inner sep=0.75pt]   [align=left] {2};
        \draw (157.33,67.5) node [anchor=north west][inner sep=0.75pt]   [align=left] {$\displaystyle x_{3}$};
        \draw (128.33,108) node [anchor=north west][inner sep=0.75pt]   [align=left] {1};
        \draw (193.33,109) node [anchor=north west][inner sep=0.75pt]   [align=left] {3};

        \end{tikzpicture}

        \caption{Type 3.}
        \label{fig:1.3}
    \end{subfigure}

    \vspace{1em}

    \begin{subfigure}[b]{0.375\linewidth}

        \tikzset{every picture/.style={line width=0.70pt}} 

        \begin{tikzpicture}[x=0.8pt,y=0.8pt,yscale=-1,xscale=1]

        \draw  [fill={rgb, 255:red, 0; green, 0; blue, 0 }  ,fill opacity=1 ] (59.8,130.1) .. controls (59.8,128.9) and (60.77,127.93) .. (61.96,127.93) .. controls (63.16,127.93) and (64.13,128.9) .. (64.13,130.1) .. controls (64.13,131.3) and (63.16,132.27) .. (61.96,132.27) .. controls (60.77,132.27) and (59.8,131.3) .. (59.8,130.1) -- cycle ;
        \draw    (60.96,128.93) -- (185.61,217.82) ;
        \draw  [fill={rgb, 255:red, 0; green, 0; blue, 0 }  ,fill opacity=1 ] (182.44,216.65) .. controls (182.44,215.45) and (183.41,214.48) .. (184.61,214.48) .. controls (185.81,214.48) and (186.78,215.45) .. (186.78,216.65) .. controls (186.78,217.85) and (185.81,218.82) .. (184.61,218.82) .. controls (183.41,218.82) and (182.44,217.85) .. (182.44,216.65) -- cycle ;
        \draw    (90.39,217.82) -- (138,74) ;
        \draw  [fill={rgb, 255:red, 0; green, 0; blue, 0 }  ,fill opacity=1 ] (135.83,76.17) .. controls (135.83,74.97) and (136.8,74) .. (138,74) .. controls (139.2,74) and (140.17,74.97) .. (140.17,76.17) .. controls (140.17,77.36) and (139.2,78.33) .. (138,78.33) .. controls (136.8,78.33) and (135.83,77.36) .. (135.83,76.17) -- cycle ;
        \draw    (137,74) -- (183.44,215.65) ;
        \draw    (60.96,129.93) -- (215.04,129.93) ;
        \draw    (90.39,217.82) -- (215.04,128.93) ;
        \draw  [fill={rgb, 255:red, 0; green, 0; blue, 0 }  ,fill opacity=1 ] (88.22,217.82) .. controls (88.22,216.62) and (89.19,215.65) .. (90.39,215.65) .. controls (91.59,215.65) and (92.56,216.62) .. (92.56,217.82) .. controls (92.56,219.01) and (91.59,219.98) .. (90.39,219.98) .. controls (89.19,219.98) and (88.22,219.01) .. (88.22,217.82) -- cycle ;
        \draw  [fill={rgb, 255:red, 0; green, 0; blue, 0 }  ,fill opacity=1 ] (211.87,130.1) .. controls (211.87,128.9) and (212.84,127.93) .. (214.04,127.93) .. controls (215.23,127.93) and (216.2,128.9) .. (216.2,130.1) .. controls (216.2,131.3) and (215.23,132.27) .. (214.04,132.27) .. controls (212.84,132.27) and (211.87,131.3) .. (211.87,130.1) -- cycle ;

        \draw (79.83,221) node [anchor=north west][inner sep=0.75pt]   [align=left] {$\displaystyle x_{1}$};
        \draw (180.33,222.5) node [anchor=north west][inner sep=0.75pt]   [align=left] {$\displaystyle x_{2}$};
        \draw (103.33,139) node [anchor=north west][inner sep=0.75pt]   [align=left] {4};
        \draw (219.33,120.5) node [anchor=north west][inner sep=0.75pt]   [align=left] {$\displaystyle x_{3}$};
        \draw (113.33,171) node [anchor=north west][inner sep=0.75pt]   [align=left] {1};
        \draw (151.33,172) node [anchor=north west][inner sep=0.75pt]   [align=left] {3};
        \draw (140.33,59.5) node [anchor=north west][inner sep=0.75pt]   [align=left] {$\displaystyle x_{4}$};
        \draw (41.33,125.5) node [anchor=north west][inner sep=0.75pt]   [align=left] {$\displaystyle x_{5}$};
        \draw (161.33,139) node [anchor=north west][inner sep=0.75pt]   [align=left] {2};
        \draw (134.33,117) node [anchor=north west][inner sep=0.75pt]   [align=left] {5};

        \end{tikzpicture}

        \caption{Type 4.}
        \label{fig:1.4}
    \end{subfigure}
    \hfill
    \begin{subfigure}[b]{0.4\linewidth}

        \tikzset{every picture/.style={line width=0.70pt}} 

        \begin{tikzpicture}[x=0.8pt,y=0.8,yscale=-1,xscale=1]

        \draw  [fill={rgb, 255:red, 0; green, 0; blue, 0 }  ,fill opacity=1 ] (79.8,150.1) .. controls (79.8,148.9) and (80.77,147.93) .. (81.96,147.93) .. controls (83.16,147.93) and (84.13,148.9) .. (84.13,150.1) .. controls (84.13,151.3) and (83.16,152.27) .. (81.96,152.27) .. controls (80.77,152.27) and (79.8,151.3) .. (79.8,150.1) -- cycle ;
        \draw    (80.96,148.93) -- (205.61,237.82) ;
        \draw  [fill={rgb, 255:red, 0; green, 0; blue, 0 }  ,fill opacity=1 ] (202.44,236.65) .. controls (202.44,235.45) and (203.41,234.48) .. (204.61,234.48) .. controls (205.81,234.48) and (206.78,235.45) .. (206.78,236.65) .. controls (206.78,237.85) and (205.81,238.82) .. (204.61,238.82) .. controls (203.41,238.82) and (202.44,237.85) .. (202.44,236.65) -- cycle ;
        \draw    (110.39,237.82) -- (158,94) ;
        \draw  [fill={rgb, 255:red, 0; green, 0; blue, 0 }  ,fill opacity=1 ] (155.83,96.17) .. controls (155.83,94.97) and (156.8,94) .. (158,94) .. controls (159.2,94) and (160.17,94.97) .. (160.17,96.17) .. controls (160.17,97.36) and (159.2,98.33) .. (158,98.33) .. controls (156.8,98.33) and (155.83,97.36) .. (155.83,96.17) -- cycle ;
        \draw    (157,94) -- (203.44,235.65) ;
        \draw    (80.96,149.93) -- (235.04,149.93) ;
        \draw    (110.39,237.82) -- (235.04,148.93) ;
        \draw  [fill={rgb, 255:red, 0; green, 0; blue, 0 }  ,fill opacity=1 ] (108.22,237.82) .. controls (108.22,236.62) and (109.19,235.65) .. (110.39,235.65) .. controls (111.59,235.65) and (112.56,236.62) .. (112.56,237.82) .. controls (112.56,239.01) and (111.59,239.98) .. (110.39,239.98) .. controls (109.19,239.98) and (108.22,239.01) .. (108.22,237.82) -- cycle ;
        \draw  [fill={rgb, 255:red, 0; green, 0; blue, 0 }  ,fill opacity=1 ] (231.87,150.1) .. controls (231.87,148.9) and (232.84,147.93) .. (234.04,147.93) .. controls (235.23,147.93) and (236.2,148.9) .. (236.2,150.1) .. controls (236.2,151.3) and (235.23,152.27) .. (234.04,152.27) .. controls (232.84,152.27) and (231.87,151.3) .. (231.87,150.1) -- cycle ;
        \draw    (81.96,150.1) -- (158,96.33) ;
        \draw    (204.61,234.48) -- (234.04,152.27) ;
        \draw    (160.17,97.17) -- (234.04,150.1) ;
        \draw    (81.96,150.1) -- (110.39,235.65) ;
        \draw    (110.39,237.82) -- (202.44,236.65) ;

        \draw (99.83,241) node [anchor=north west][inner sep=0.75pt]   [align=left] {$\displaystyle x_{1}$};
        \draw (200.33,242.5) node [anchor=north west][inner sep=0.75pt]   [align=left] {$\displaystyle x_{2}$};
        \draw (123.33,159) node [anchor=north west][inner sep=0.75pt]   [align=left] {4};
        \draw (239.33,140.5) node [anchor=north west][inner sep=0.75pt]   [align=left] {$\displaystyle x_{3}$};
        \draw (133.33,191) node [anchor=north west][inner sep=0.75pt]   [align=left] {1};
        \draw (171.33,192) node [anchor=north west][inner sep=0.75pt]   [align=left] {3};
        \draw (160.33,79.5) node [anchor=north west][inner sep=0.75pt]   [align=left] {$\displaystyle x_{4}$};
        \draw (61.33,145.5) node [anchor=north west][inner sep=0.75pt]   [align=left] {$\displaystyle x_{5}$};
        \draw (181.33,159) node [anchor=north west][inner sep=0.75pt]   [align=left] {2};
        \draw (154.33,137) node [anchor=north west][inner sep=0.75pt]   [align=left] {5};
        \draw (158.42,240.23) node [anchor=north west][inner sep=0.75pt]   [align=left] {1};
        \draw (221.32,196.38) node [anchor=north west][inner sep=0.75pt]   [align=left] {1};
        \draw (198.33,110) node [anchor=north west][inner sep=0.75pt]   [align=left] {1};
        \draw (108.33,109) node [anchor=north west][inner sep=0.75pt]   [align=left] {1};
        \draw (83.33,191) node [anchor=north west][inner sep=0.75pt]   [align=left] {1};

        \end{tikzpicture}

        \caption{Type 5.}
        \label{fig:1.5}
    \end{subfigure}

    \caption{Five types of graphs.}
    \label{fig:1}
\end{figure}
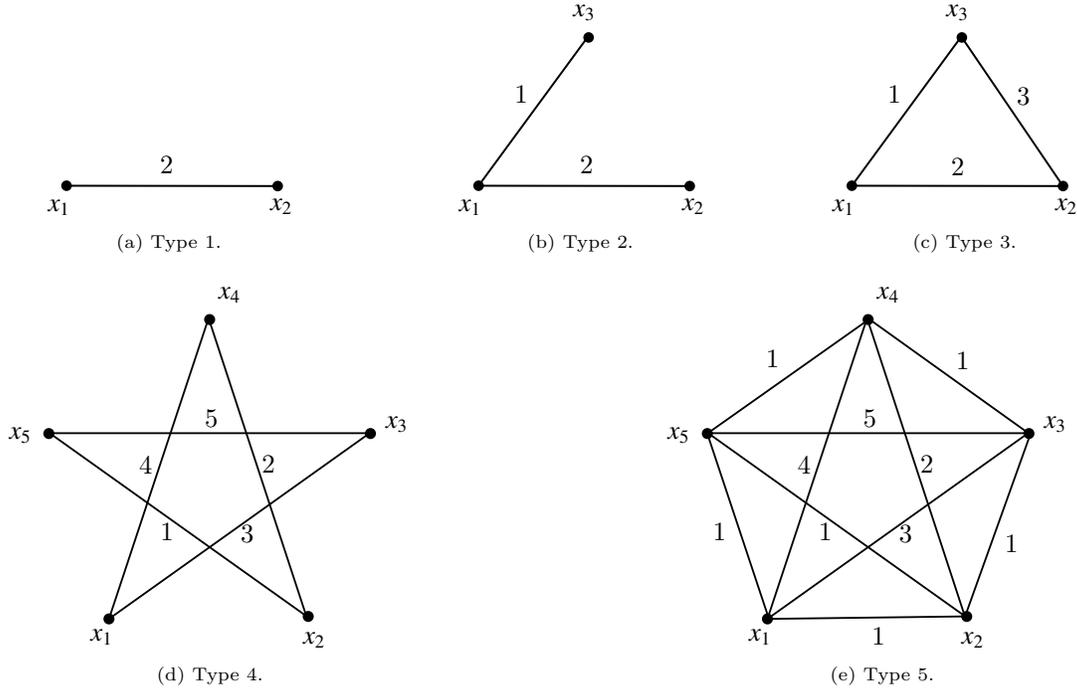

\begin{figure}
    \centering
    \begin{subfigure}[b]{0.49\linewidth}
        \centering
        \includegraphics[width=\linewidth]{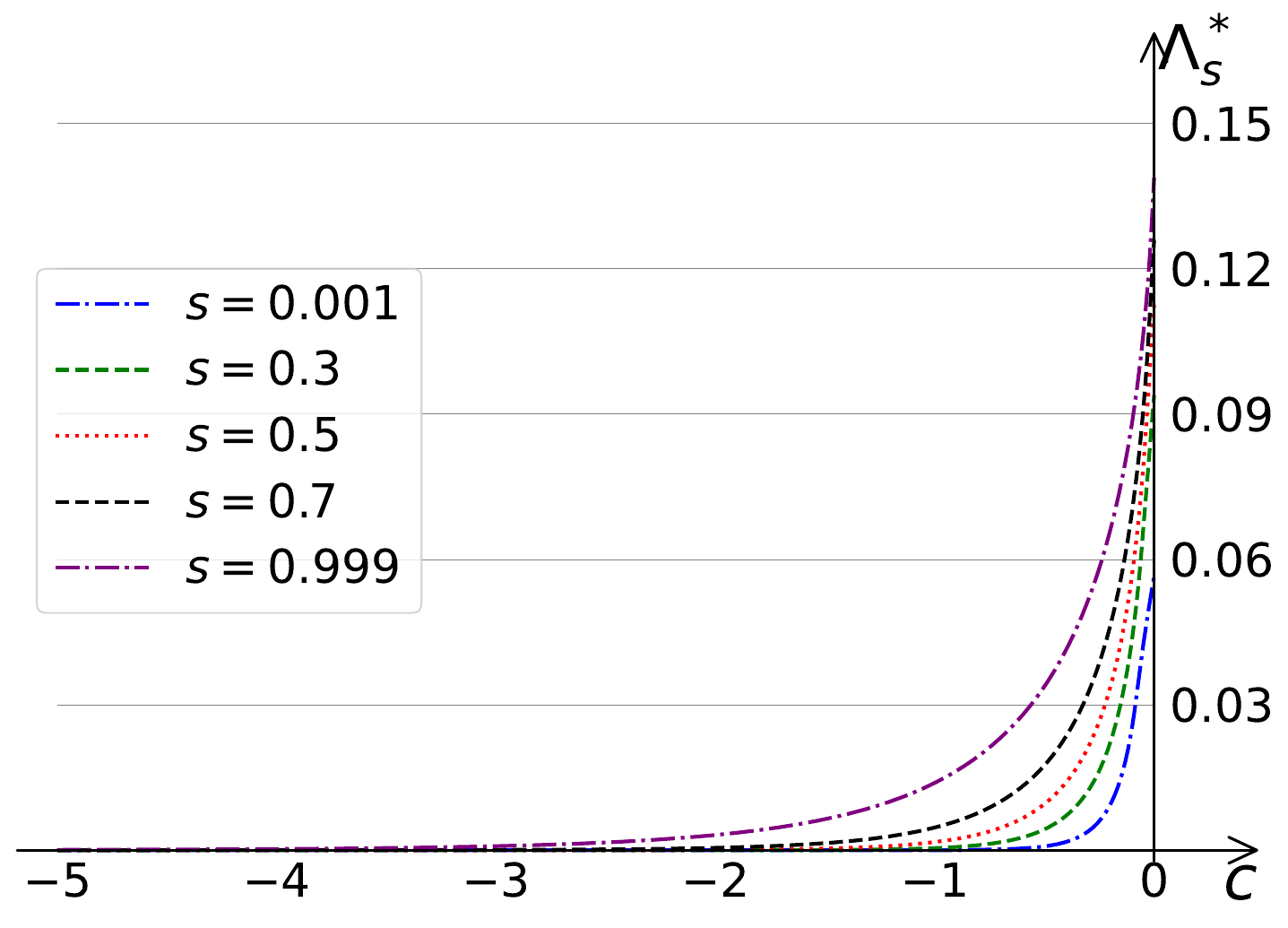}
        \caption{Variation of $\Lambda_s^*$ for $c$ from $-5$ to $0$ in Type 1.}
        \label{fig:2.1}
    \end{subfigure}

    \begin{subfigure}[b]{0.49\linewidth}
        \centering
        \includegraphics[width=\linewidth]{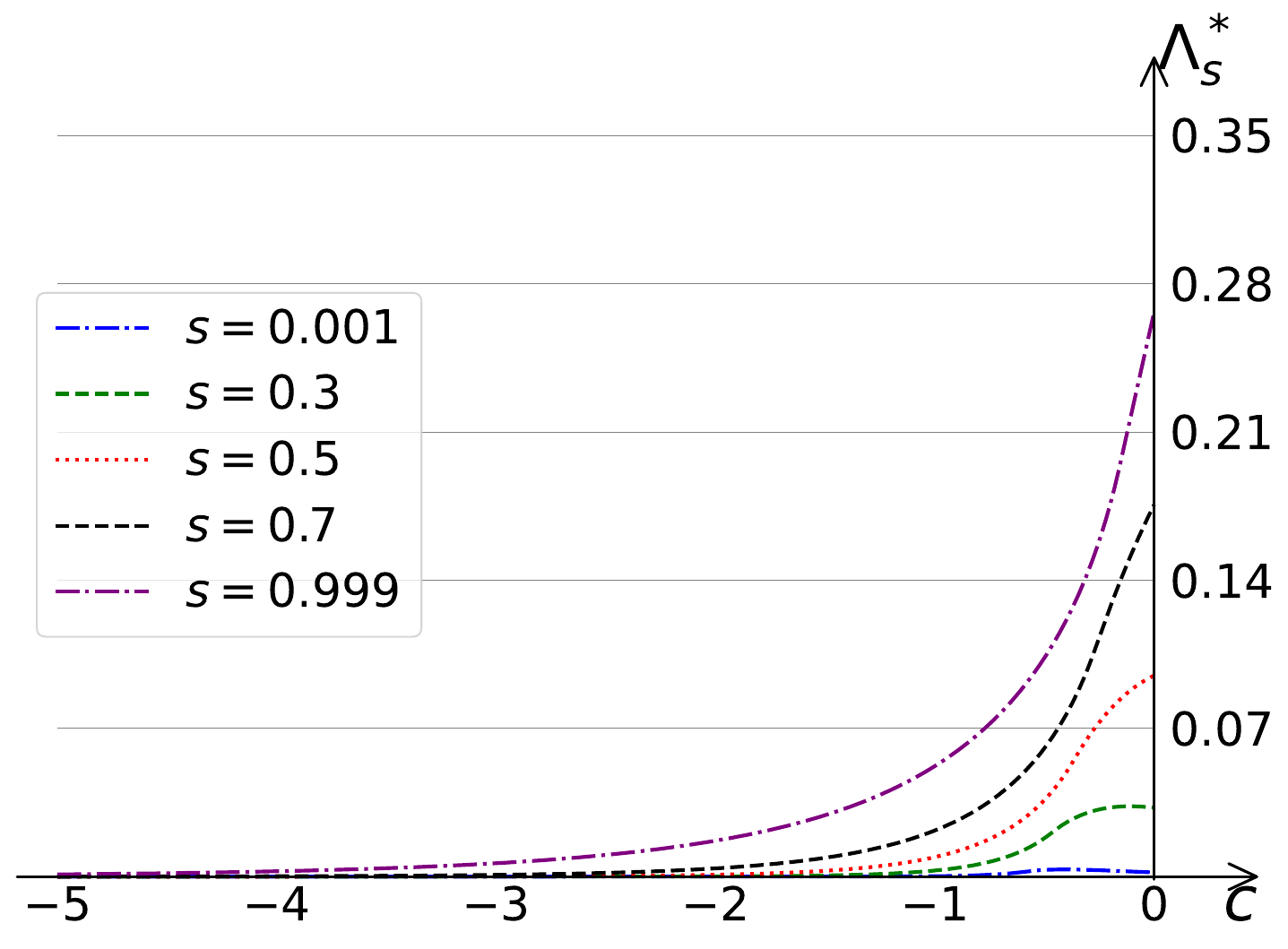}
        \caption{Variation of $\Lambda_s^*$ for $c$ from $-5$ to $0$ in Type 2.}
        \label{fig:2.2}
    \end{subfigure}
    \hfill
    \begin{subfigure}[b]{0.49\linewidth}
        \centering
        \includegraphics[width=\linewidth]{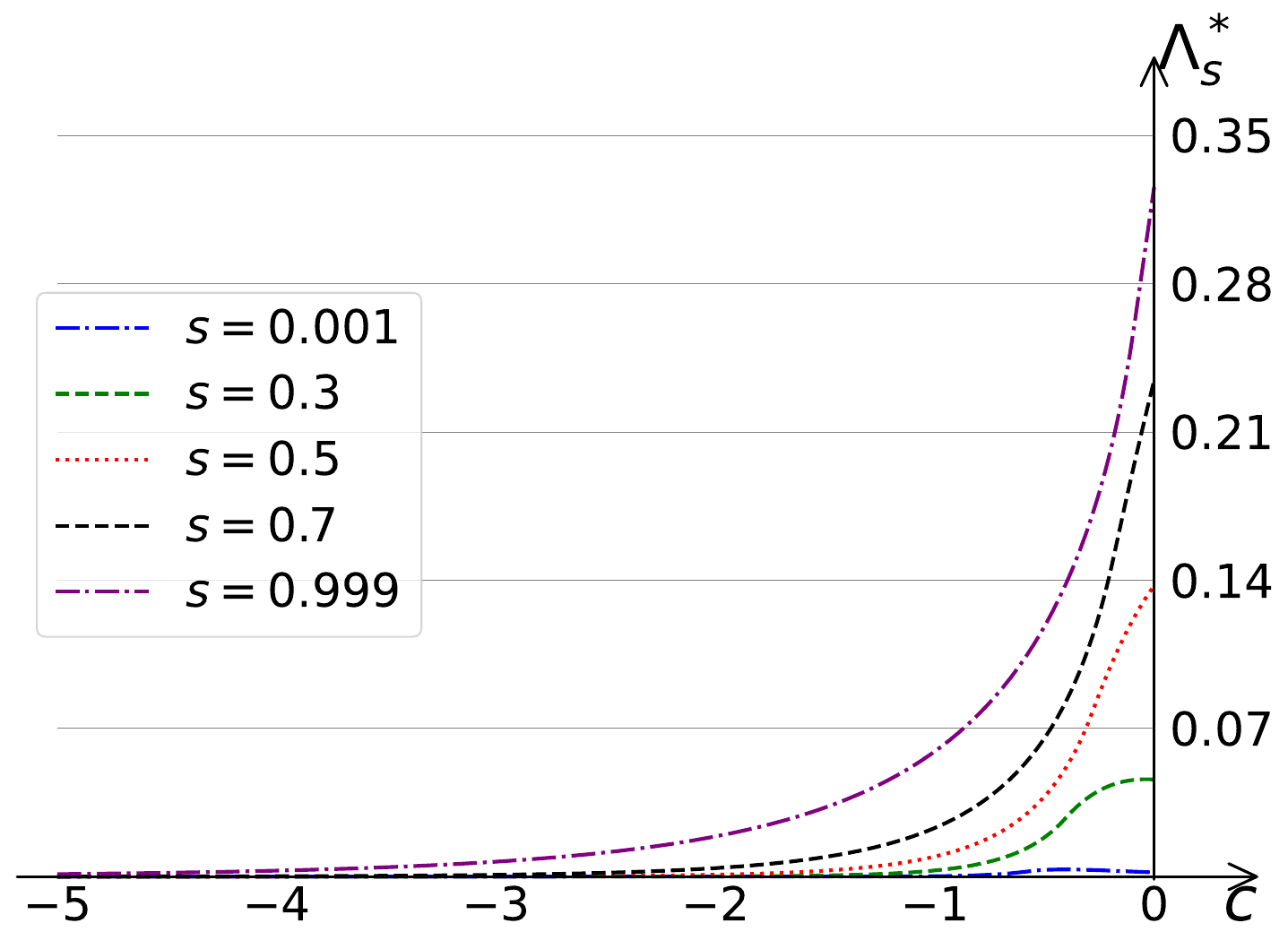}
        \caption{Variation of $\Lambda_s^*$ for $c$ from $-5$ to $0$ in Type 3.}
        \label{fig:2.3}
    \end{subfigure}

    \begin{subfigure}[b]{0.49\linewidth}
        \centering
        \includegraphics[width=\linewidth]{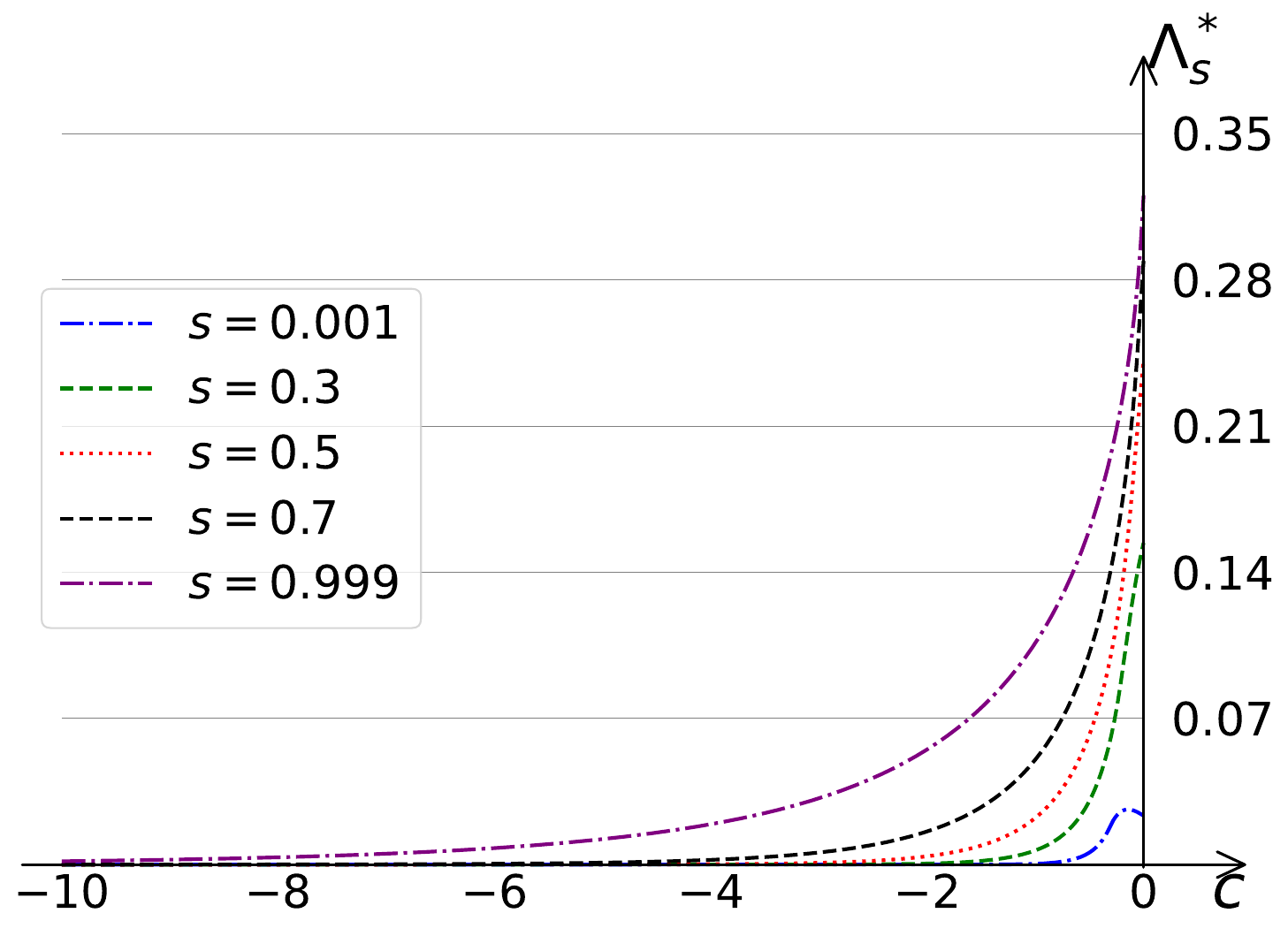}
        \caption{Variation of $\Lambda_s^*$ for $c$  from $-10$ to $0$ in Type 4.}
        \label{fig:2.4}
    \end{subfigure}
    \hfill
    \begin{subfigure}[b]{0.49\linewidth}
        \centering
        \includegraphics[width=\linewidth]{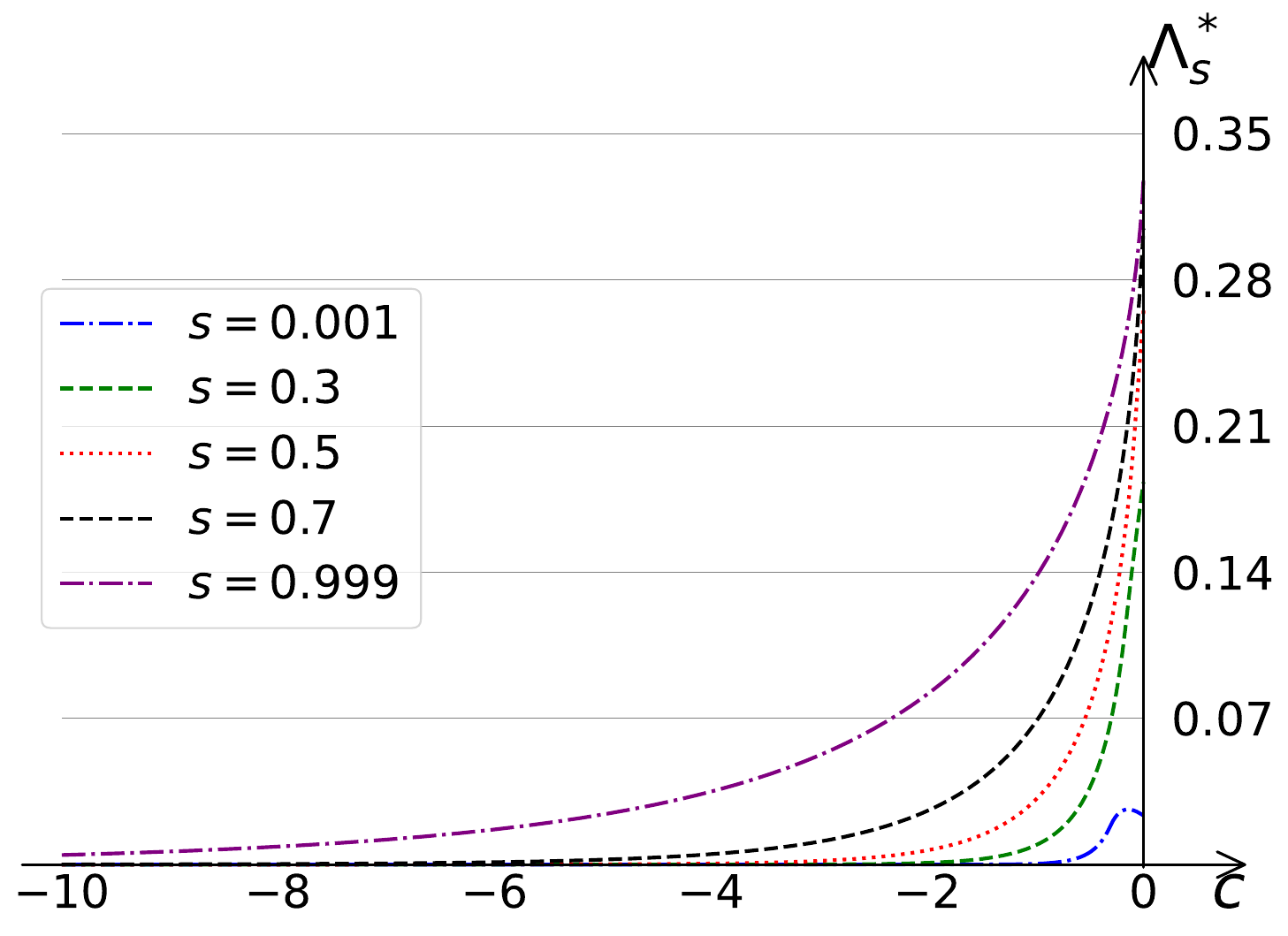}
        \caption{Variation of $\Lambda_s^*$ for $c$ from $-10$ to $0$ in Type 5.}
        \label{fig:2.5}
    \end{subfigure}

    \caption{The experimental results.}
    \label{fig:2}
\end{figure}

For each type, we calculated the variation of $ \Lambda_s^* $ as $ c<0 $ decreasing under different values of $ s $. By visualizing the results, we obtained the outcomes shown in Figure \ref{fig:2}. In Figure \ref{fig:2}, the horizontal axis represents the values of $ c $, with a step size of 0.01, and the vertical axis represents the values of $ \Lambda_s^* $. For each type, we calculate $ \Lambda_s^* $ on  five different $ s $ values, specifically $ s$ is equal to $0.001, 0.3, 0.5, 0.7, 0.999$. It is worth noting that the vertical range in Figure \ref{fig:2.1} is 0.15, unlike the other four plots, which have a vertical range of 0.35. Additionally, the horizontal range in Figures \ref{fig:2.4} and \ref{fig:2.5} is from -10 to 0, in contrast to the range of -5 to 0 used in the first three plots.

 By examining the content of all five plots, it can be observed that the lines in each plot do not intersect and are arranged from top to bottom in descending order of $ s $ values. This indicates that as $ s $ increases, the value of $ \Lambda_s^* $ also rises. By comparing Figures \ref{fig:2.2} and \ref{fig:2.3}, it can be observed that complete graphs exhibit higher $ \Lambda_s^* $ values than non-complete graphs. Although Figures \ref{fig:2.4} and \ref{fig:2.5} appear visually similar, the numerical results clearly show that the $ \Lambda_s^* $ values are higher for complete graphs.

 By examining the green dashed lines in Figures \ref{fig:2.2} and \ref{fig:2.3}, as well as the blue dash-dotted lines in Figures \ref{fig:2.2}, \ref{fig:2.3}, \ref{fig:2.4}, and \ref{fig:2.5}, it can be observed that the value of $ \lambda_s^* $ is not always positively correlated with $ c $. Under certain conditions, $ \lambda_s^* $ becomes negatively correlated with $ c $. For example, in Figure \ref{fig:2.4}, this negative correlation is evident in the portion of the blue dash-dotted line near the zero region on the horizontal axis. An interesting direction for future research would be to investigate the conditions under which $ \Lambda_s^* $ becomes negatively correlated with $ c $.
Another important conclusion that can be drawn from the experimental results is that when the value of $ c $ becomes sufficiently small, $ \Lambda_s^* $ converges to 0.
Another intriguing research direction would be to theoretically prove whether $ \Lambda_s^* $ converges to 0 as $ c \rightarrow -\infty $.

	\bigskip

\nolinenumbers
	
	\noindent
	\textbf{Acknowledgements.}   All authors would like to appreciate the reviewers and editors for their careful reading, constructive comments and helpful suggestions on this paper.\\		

		\noindent
	\textbf{Author Contributions.} All authors contributed equally to this work.  \\	

		\noindent
	\textbf{Funding.} This paper is supported by the National Natural Science Foundation of China (Grant Number: 12471088).


\end{document}